\documentclass[12pt]{amsart}

\usepackage{cancel}
\usepackage{latexsym, amssymb, amsmath,esint}
\usepackage{soul}
\usepackage{amsfonts, graphicx}
\usepackage{graphicx,color}
\newcommand{\be}{\begin{equation}}
\newcommand{\ee}{\end{equation}}
\newcommand{\beq}{\begin{eqnarray}}
\newcommand{\eeq}{\end{eqnarray}}

\newtheorem{thm}{Theorem}[section]

\newtheorem{lma}{Lemma}[section]

\newtheorem{cor}{Corollary}[section]
\newtheorem{defn}{Definition}[section]
\theoremstyle{remark}
\newtheorem{rem}{Remark}[section]
\numberwithin{equation}{section}

\def\be{\begin{equation}}
\def\ee{\end{equation}}
\def\bee{\begin{equation*}}
\def\eee{\end{equation*}}
\def\ol{\overline}
\def\lf{\left}
\def\ri{\right}

\def\Ric{\text{\rm Ric}}
\def\Rm{\text{\rm Rm}}

\def\cR{ \mathcal{R}}

\def\wh{\widehat}
\def\wt{\widetilde}

\def\p{\partial}

\def\ol{\overline}
\def\heat{\lf(\frac{\p}{\p t}-\Delta_t\ri)}
\def\tr{\operatorname{tr}}

\def\e{\varepsilon}

\def\a{{\alpha}}
\def\b{{\beta}}

\def\R{\mathbb{R}}

\begin{document}

\title[]
{Rigidity of Lipschitz map using harmonic map heat flow}

 \author{Man-Chun Lee$^1$}
\address[Man-Chun Lee]{Department of Mathematics, The Chinese University of Hong Kong, Shatin, Hong Kong, China}
\email{mclee@math.cuhk.edu.hk}

\author{Luen-Fai Tam$^2$}
\address[Luen-Fai Tam]{The Institute of Mathematical Sciences and Department of Mathematics, The Chinese University of Hong Kong, Shatin, Hong Kong, China.}
 \email{lftam@math.cuhk.edu.hk}

 \thanks{$^1$Research partially supported by Hong Kong RGC grant (Early Career Scheme) of Hong Kong No. 24304222 and a direct grant of CUHK}
\thanks{$^2$Research partially supported by Hong Kong RGC General Research Fund \#CUHK 14301517.}

\renewcommand{\subjclassname}{
  \textup{2020} Mathematics Subject Classification}

\date{\today}

\begin{abstract}
Motivated by the Lipschitz rigidity problem in scalar curvature geometry,  we prove that if a closed smooth spin manifold admits a distance decreasing continuous map of non-zero degree to a sphere,  then either the scalar curvature is strictly less than the sphere somewhere or the map is a distance isometry.  Moreover, the property also holds for continuous metrics with scalar curvature lower bound in some weak sense. This extends a result in the recent work of Cecchini-Hanke-Schick \cite{CecchiniHankeSchick2022} and answers a question of Gromov.  The method is based on studying the harmonic map heat flow coupled with the Ricci flow from rough initial data to reduce the case  to smooth metrics and smooth maps so that results by Llarull \cite{Llarull1998} can be applied.
\end{abstract}

\maketitle

\markboth{Man-Chun Lee, Luen-Fai Tam}{Rigidity of Lipschitz map to sphere}
\section{introduction}

In recent years,  there has been many exciting development in understanding the scalar curvature of a Riemannian manifold.  We refer readers to Gromov's lecture note \cite{GromovLecture} on scalar curvature for a comprehensive overview.  In \cite{Llarull1998}, Llarull proved the following striking result which confirms one of Gromov's conjectures.
\begin{thm}[\cite{Llarull1998}]
Let $M^n$ be a compact spin manifold.  If $f:(M,g)\to (\mathbb{S}^n,g_{sphere})$ is a distance decreasing map of non-zero degree into the unit sphere in $\R^{n+1}$ with standard metric and if the scalar curvature  $\mathcal{R}(g)$ of $g$ is greater than or equal to  $n(n-1)$, then $f$ must be an isometry.
\end{thm}
In fact, Llarull proved the Theorem under  a weaker assumption that $f$ is so-called $(1,
\Lambda^2)$-contracting. This means that the map $\Lambda^2df_x$ from $\Lambda^2(T_x(M))$ to $\Lambda^2(T_{f(x)}\mathbb{S}^n)$ induced by $f$ has norm less than or equal to 1. If we only assume the map is $(1,\Lambda^k)$-contracting for $k\ge 3$, then the result is not true \cite{Llarull1998}.
The theorem   was later generalized by Goette-Semmelmann \cite{GoetteSemmelmann2002} to manifolds with non-zero Euler characteristic and non-negative curvature operator in place of the sphere. The method is based on the study of twisted Dirac operators. It is then asked by Gromov whether the rigidity still holds if $f$ is only distance decreasing continuous map, \cite[Section 4.5, question (b)]{GromovLecture}. Since distance can be defined even if we only assume that the metric $g$ is $C^0$ \cite{CecchiniHankeSchick2022},  one may ask whether the result is still true when $g$ is $C^0$ and $f$ is only distance decreasing.  In this case, we need to define  a notion of scalar curvature lower bound on metrics which are only $C^0$.

In this direction,  Cecchini-Hanke-Schick \cite{CecchiniHankeSchick2022} prove that Lllarull's theorem is still true in the sense that $f$ is a metric isometry, under the following assumptions:  (i) $g$ is only $W^{1,p}$ for $p>n$; (ii) the scalar curvature is bounded below by $n(n-1)$ in the distribution sense introduced by  Lee-LeFloch \cite{LeeLeFloch2015}; (iii) $f$ is a Lipschitz map which is distance non-increasing or more generally is $(1,\Lambda^2)$-contracting; (iv) $n$ is even.

First note that assumption (i) implies that $g$ is $C^0$ by Sobolev embedding. Moreover, $f$ is Lipschitz implies that $df$ can be defined almost everywhere and we still can talk about $(1,\Lambda^2)$-contracting. As in  \cite{Llarull1998}, Cecchini-Hanke-Schick use Dirac operator under these weaker conditions on regularity. They conjecture that the result should be true even if dimension $M$ is odd.

Motivated by the question of Gromov and the work of Cecchini-Hanke-Schick,  in this work, we consider the problem for general dimension $n$ and metric $g$ with only $C^0$ regularity.  With $C^0$ metric structure, the Lipschitz constant of a continuous map can still be defined:
\begin{defn}
Let $M$ and $N$ be two smooth manifolds. Suppose $g$ and $h$ are continuous metrics on $M$ and $N$ respectively. A continuous map $f:(M,d_g)\to (N,d_h)$ is said to be $\Lambda$-Lipschitz if
$$\mathrm{Lip}_{g,h}(f)=\sup\left\{ \frac{d_h\left( f(x),f(y)\right)}{d_g(x,y)}: x,y\in M, x\neq y\right\}\leq \Lambda.$$
We say that $f$ is Lipschitz continuous if $f$ is $\Lambda$-Lipschitz for some $\Lambda>0$. Here the Riemannian distance $d_{\hat g}$ of a continuous metric $\hat g$ is defined by minimizing the length of smooth regular curve between points with respect to $\hat g$, see \cite[Reminder 1.2]{CecchiniHankeSchick2022} and \cite[(2.2)]{Burtscher2015}.
\end{defn}
\begin{rem}
When both $g,h$ and $f$ are smooth, then the map $f$ being $\Lambda$-Lipschitz implies that $$f^*h\leq \Lambda^2 \cdot g$$
on $M$ by letting $x\to y$.  The converse holds trivially by integrating along geodesic. In the non-smooth case, it still holds almost everywhere, for instance see \cite[Proposition 2.1]{CecchiniHankeSchick2022}, but we will not rely on this fact in this work.
\end{rem}

The definition of Lipschitz continuity of a map with respect to a continuous metric $g$ can be defined naturally while the scalar curvature lower bound is very subtle. We use the following:
\begin{defn}\label{defn-C0-scalar}
Let $g_0$ be a $C^0$ metric on a compact manifold $M$. We say that $g_0$ is of $\mathcal{R}(g_0)\geq \sigma_0$ if there exists a sequence of smooth metric $g_i$ on $M$ such that $\mathcal{R}(g_i)\geq \sigma_0$ on $M$ and $g_i\to g_0$ in $C^0(M)$ as $i\to +\infty$.
\end{defn}
The definition can be rephrased as follows: Let $G_{\sigma_0}$ be the subset of $C^2$ metrics with scalar curvature bounded from below by $\sigma_0$. A metric $g\in C^0(M)$ is said to have scalar curvature bounded below by $\sigma_0$ if $g\in \ol G_{\sigma_0}$, where the closure is taken in terms of $C^0$ norm.  This definition is natural by Gromov \cite{Gromov2014} and Bamler \cite{Bamler2016} in the sense that if $g\in \ol G_{\sigma_0}$ and if  $g$ is $C^2$, then the scalar curvature of $g$ is bounded below by $\sigma_0$ in the usual sense.

Suppose $g$ is a $C^0$ metric. If $g$ is satisfies one of the following, then it is known that it also has scalar curvature bounded below by $\sigma$ in the sense of  Definition~\ref{defn-C0-scalar}:   (a) $g$ has scalar curvature bounded below by $\sigma$ in the sense of Burkhardt-Guim \cite{Burkhardt2019} using regularizing Ricci flow; (b) $g\in W^{1,p},p>n$ with $\mathcal{R}\geq \sigma$ in the sense of distribution as in \cite{LeeLeFloch2015} by Jiang-Sheng-Zhang \cite{JiangShengZhang2021} ; (c) $g$ is smooth away from singularity $\Sigma$ of co-dimension at least three and has $\mathcal{R}(g)\geq \sigma$ outside $\Sigma$ by the work \cite{LeeTam2021} of the authors; (d) there exist smooth $g_i\to g$ in $C^0$ norm so that some integral form of  lower bound of $\mathcal{R}(g_i)$ is satisfied by Huang and the first named author \cite{HuangLee2021}.

Under this general notion of scalar curvature lower bound, we obtain the following:
\begin{thm}\label{intro-mainTHM}
Let $M^n$ be a compact Riemannian spin manifold of dimension $n$ and $g_0$ is a $C^0$ metric on $M$ with $\mathcal{R}(g_0)\geq n(n-1)$ in the sense of definition~\ref{defn-C0-scalar}. Suppose there is $1$-Lipschitz continuous map $f:(M,d_{g_0})\to \mathbb{S}^n$ with non-zero degree, then $f$ is a distance isometry.
\end{thm}
In particular, this gives a complete answer to \cite[Section 4.5, question (b)]{GromovLecture} for all dimension $n$. We also confirm the conjecture    in \cite{CecchiniHankeSchick2022} that their result is true for the odd dimensional cases, assuming that the map is 1-Lipschitz. However, we are unable to prove the result under the weaker condition that the map is $(1,\Lambda^2)$-contracting.

 Our proof is to reduce the cases to the smooth cases and apply Llarull's theorem. Since in this case, the metric and the map are both non-smooth, we will use Ricci-DeTurck flow and the results of M. Simon \cite{Simon2002} to regularize the metric and use a smooth map to regularize the Lipschitz map. In order to use Llarull's result, the crucial point is  a new monotonicity for harmonic heat flow coupled with the Ricci flow.
 Moreover, using the geometric flow approach, if in addition $g_0$ is of better regularity as considered in \cite{CecchiniHankeSchick2022}, one can obtain a slightly stronger conclusion on the regularity of $g_0$, see Theorem~\ref{Thm:better-regular}.

Llarull's results have been generalized in another direction. Namely maps between manifolds with boundary.    It was previously considered by Lott \cite{Lott2021} using boundary value problems for Dirac operators, see also \cite{BaerHanke2020,CecchiniZeidler2021} for related works.  When $n=3$, this was also considered by Hu-Liu-Shi \cite{HuLiuShi2022} using $\mu$-bubbles where the boundary conditions can be further relaxed in contrast with method of Dirac operator. In this regard,  using gluing method with Theorem~\ref{intro-mainTHM}, we have a relatively simple rigidity result for  domains inside sphere which holds for all dimensions:
\begin{cor}\label{c-domain}
Suppose $\Omega$ be a domain inside the standard sphere $(\mathbb{S}^n,h)$ with smooth boundary.  If $g_0$ is a smooth metric on $\Omega$ such that
\begin{enumerate}
\item[(i)] $g_0\geq h$ where $h$ is the standard spherical metric;
\item[(ii)] $\mathcal{R}(g_0)\geq n(n-1)$ on $\Omega$;
\item[(iii)] $H(g)\geq H(h)$ on $\partial\Omega$, where $H(g), H(h)$ are the mean curvatures with respect to the unit outward normals and with respect to $g, h$;
\item [(iv)] $g_0=h$ on $\partial \Omega$,
\end{enumerate}
then $g_0=h$ on $\Omega$. Moreover,  if $\Omega$ is the hemisphere, then the same conclusion holds without assumption (iv).
\end{cor}

   One may indeed obtain an analogous statement as in \cite[Theorem B]{CecchiniHankeSchick2022} using the same strategy. We work on the sphere only to illustrate the use of Theorem \ref{intro-mainTHM}.

\medskip

The paper is organized as follows. In Section~\ref{Sec: RF estimate},  we review some basic definition of harmonic map heat flow coupled with the Ricci flow and establish a new monotonicity. In Section~\ref{Sec: HMHF-roughtData}, we construct the harmonic map heat flow coupled with a smooth Ricci flow starting from Lipschitz initial data and obtain estimates under scaling invariant curvature control. In Section~\ref{Sec:Proof-Main}, we will give the proof of main results.
\section{Monotonicity along harmonic heat flow}\label{Sec: RF estimate}

In this section, we will prove that some quantities will be monotone along the harmonic heat flow coupled with the Ricci flow on compact manifolds. Using these, we can regularize the Lipschitz map and metric while preserving the rigidity structure.  Let us first recall the notion of harmonic map heat flow.

Suppose $f: (M,g)\to (N,h)$ is a smooth map, then $df$ is a section of  $T^*(M)\otimes f^{-1}(T(N))$ where $f^{-1}(T(N))$ is the pull-back bundle by $f$. Let $D$ be the covariant derivative induced by the Riemannian connections of $g, h$. Then the second fundamental form $Ddf$ is a section of $  T^*(M)\otimes T^*(M)\otimes f^{-1}(T(N))$. The trace $\tau(f)$ of $Ddf$ with respect to $g$ is called the tension field which is a vector field along $f$. The energy density $e(f)$ of the map is $|df|^2$ where the inner product is taken with respect the metric in $ T^*(M)\otimes f^{-1}(T(N))$ induced by $g, h$. Equivalently,  we have $e(f)=\tr_{g}f^*(h)$.

If $g(t)$ is a smooth family of metrics on $M$, then the harmonic map heat flow is given by
\be\label{e-harmonicflow}
\frac{\p}{\p t}F=\tau(F)
\ee
where $F: M\times[0,T]\to N$ and $\tau(F)$ at time $t$ is the tension field of $F(\cdot,t)$ with respect to the metric $g(t)$ in the domain.  We note that $\frac{\p}{\p t}F$ is indeed $F_*(\frac{\p}{\p t})$. We want to discuss the behaviours of the eigenvalues of $(F(t))^*(h)$ with respect to $g(t)$.

\begin{thm}\label{prop:metric-lowerBound}
Let $(M^n,g_0)$ and $(N^n,h)$ be two compact manifolds such that $f^*h\leq \Lambda^2 g_0$ on $M$ for some $\Lambda>0$.  Suppose $g(t)$ is a smooth family of metrics   on $M\times [0,T]$ satisfying
$$
\p_t g_{ij}=2k_{ij}
$$
such that $k+\Ric(g)\ge0$.
Let $F(t): (M,g(t))\to (N,h)$ be a family of smooth map satisfying the harmonic map heat flow:
\begin{equation}
\partial_t F =\tau( F),\;\;
F(0)=f.
\end{equation}
  If either one of the followings hold:
  \begin{enumerate}
  \item[(i)]  the sectional curvature of $h$ satisfies ${K}(h)\leq \kappa$ for some $\kappa\ge0$ or;
  \item[(ii)]the sectional curvature of $h$ is non-negative and $\Ric(h)\le (n-1)\kappa$,
  \end{enumerate}
  then
$$ (\Lambda^{-2}-2(n-1)\kappa t)F^*h\leq g(t),$$
on $M\times [0,T]$.

\end{thm}

\begin{proof} It is sufficient to prove the Theorem for those $t$ with   $\Lambda^{-2}-2(n-1)t>0$.
Let $H=F^*(h)$ so that $H_{ij}(x)=F^\a_i F^\b_j h_{\a\b}(F(x))$ for $x\in M$. Hence
\be\label{e-time}
\p_t H_{ij}=\lf(F^\a_{it}F^\b_j+F^\a_{i}F^\b_{tj}\ri)h_{\a\b},
\ee
because the derivative of $h$ is zero with respect to the connection on the pull-back bundle $F^{-1}(T(N))$. On the other hand,
\bee
\begin{split}
\Delta_t H_{ij}=&g^{pq}(F^\a_i F^\b_j h_{\a\b} )_{;pq}\\
=&g^{pq}\lf((F^\a_i)_{|pq} F^\b_j+F^\a_i (F^\b_j)_{|pq}+2(F^\a_i)_{|p} (F^\b_j)_{|q} \ri)h_{\a\b}
\end{split}
\eee
where $\Delta_t$ is the Laplacian with respect to $g(t)$ and $_{|p}$ denotes the covariant derivative of $D$. Using $(F^\a_i)_{|pq}=(F^\a_p)_{|iq}$, the Ricci identity, and the fact that $\tau(F)^\a=g^{pq}F^\a_{p|q}$ with $g=g(t)$, we have

\bee
\begin{split}
\Delta_t H_{ij}=&   F^\b_j\lf(  (\tau(F)^\a)_{|i}+R_{i}^l F_l^\a+g^{kl}\wt R_{\gamma\sigma \delta}\,^\a F^\delta_k F^\gamma_l F^\sigma_i\ri)h_{\a\b}\\
&+F^\a_i\lf(  (\tau(F)^\b)_{|j}+R_{j}^l F_l^\b+g^{kl}\wt R_{\gamma\sigma \delta}\,^\b F^\delta_k F^\gamma_l F^\sigma_j\ri)h_{\a\b}\\
&+2g^{pq} (F^\a_i)_{|p} (F^\b_j)_{|q} h_{\a\b}\\
=&F_j^\b (\tau(F)^\a)_{|i}+F_i^\a (\tau(F)^\b)_{|j} +R_i^lH_{lj}+R_j^lH_{il}\\
&+
2g^{kl}\wt R(u_l, u_i, u_k, u_j)+2g^{pq} (F^\a_i)_{|p} (F^\b_j)_{|q} h_{\a\b}
\end{split}
\eee
where $\widetilde \Rm$ denotes the curvature of $h$ and $u_i=F_*(\p_i)$ in local coordinates $x^i$ with $\p_i=\p_{x^i}$.
Hence
\bee
\begin{split}
\heat H_{ij}=&\lf(F^\a_{it}F^\b_j+F^\a_{i}F^\b_{tj}\ri)h_{\a\b}\\
&-\bigg[F_j^\b (\tau(F)^\a)_{|i}+F_i^\a (\tau(F)^\b)_{|j} +R_i^lH_{lj}+R_j^lH_{il}+
2g^{kl}\wt R(u_l, u_i, u_k, u_j)\\
&+2g^{pq} (F^\a_i)_{|p} (F^\b_j)_{|q} h_{\a\b}\bigg]\\
=&-\bigg[ R_i^lH_{lj}+R_j^lH_{il}+
2g^{kl}\wt R(u_l, u_i, u_k, u_j)
 +2g^{pq} (F^\a_i)_{|p} (F^\b_j)_{|q} h_{\a\b}\bigg]
\end{split}
\eee
because $F^\a_{it}=F^\a_{ti}$ and $\p_t F^\a=\tau(F)^\a$ for all $i, \a$.

If we define $A(t)=\lambda(t)g -F^*h=\lambda (t)g-H$ at time $t$ for some function $\lambda(t)$, then whenever $\lambda(t)>0$, we have
\be\label{e-evolution-A}
\begin{split}
\heat A_{ij}\geq &\lambda' g_{ij}+2\lambda k_{ij}+R_i^lH_{lj}+R_j^lH_{il}+
2g^{kl}\wt R(u_l, u_i, u_k, u_j)
 \\
 &+2g^{pq} (F^\a_i)_{|p} (F^\b_j)_{|q} h_{\a\b}\\
=&:B_{ij}.
\end{split}
\ee

We want to find $\lambda(t)>0$ so that $B_{ij}$ satisfies the null-eigenvector assumption at every point $(x,t)$, $t>0$. Namely, suppose $A_{ij}w^iw^j\ge 0$ for all $w\in T_x(M)$, and  $v$ is such that $A_{ij}v^j=0$. Then we want to prove that $B_{ij}v^iv^j\ge0$.

Let $v$ be such a vector. If $v=0$, obviously, $B_{ij}v^iv^j=0$. So we may assume that $v$ is a unit vector by scaling.  At $x$, we choose a local coordinates $x^i$ so that $p_i=\p_{x^i}$ are orthonormal and $(F(t))^*h$ is diagonalized with respect to $g(t)$ at $x$.   Since   $A(w,w)\ge0$ for all $w$ and $A(v,v)=0$, $v$ is an eigenvector of $A$. We may assume that $v=\p_1$. $A(\p_i,\p_i)\ge0$ is equivalent to say that  $||F_*(\p_i)||^2_h\le \lambda$ for all $i$. Also $A_{ij}v^j=0$ implies $H_{ij}v^j=\lambda g_{ij}v^j$.
Hence
\be\label{e-B}
\begin{split}
B_{ij}v^iv^j=&\lambda' g_{ij}v^iv^j+2\lambda k_{ij}v^iv^j+R_i^lH_{lj}v^iv^j+R_j^lH_{il}v^iv^j+
2g^{kl}\wt R(u_l, u_i, u_k, u_j)v^iv^j
 \\
 &+2g^{pq} (F^\a_i)_{|p} (F^\b_j)_{|q} h_{\a\b}v^iv^j\\
 \ge&\lambda'+2\lambda(k_{ij}+R_{ij})v^iv^j-2\sum_{i=1}^n\wt R(F_* (\p_1), F_*(\p_i), F_*(\p_i), F_*(\p_1)\\
\ge &\lambda'   -2\sum_{i=1}^n\wt R(F_* (\p_1), F_*(\p_i), F_*(\p_i), F_*(\p_1),
\end{split}
\ee
because $k+\Ric\ge0$. Here $F$ refers to $F(t)$.
Since $h(F_*(\p_i), F_*(\p_j))=\b_i\delta_{ij}$ for some $\b_i\ge0$, one can find $u_i$ which are orthonormal at $F(x)$ so that $F_*(e_i)=\b_i^\frac12 u_i$. Note that $\b_i\le \lambda$.

 Hence if the sectional curvature of $h$ is bounded above by $\kappa$,  then
\bee
\sum_{i=1}^n\wt R(F_* (\p_1), F_*(\p_i), F_*(\p_i), F_*(\p_1)=    \sum_{i=1}^n\wt R(\b_1^\frac12 u_1, \b_i^\frac12 u_i ,\b_i^\frac12 u_i, \b_1^\frac12u_1)\le (n-1)\kappa \lambda^2.
\eee
If the sectional curvature of $h$ is non-negative and $\Ric(h)\le (n-1)\kappa$, then
\bee
\sum_{i=1}^n\wt R(F_* (\p_1), F_*(\p_i), F_*(\p_i), F_*(\p_1)\le\lambda^2    \sum_{i=1}^n\wt R(  u_1,   u_i ,  u_i,  u_1)\le (n-1)\kappa \lambda^2.
\eee

Putting this back to \eqref{e-B}, we have
$$
B_{ij}v^iv^j\ge \lambda'-2(n-1)\kappa \lambda^2.
$$
In either cases, if we let
$$
\lambda(t)=\lf(\Lambda^{-2}- 2(n-1) \kappa t\ri)^{-1}.
 $$
Then $\lambda(t)>0$ as long as $\Lambda^{-2}- 2(n-1) \kappa t>0$ and $B$ satisfies the null eigenvector assumption. Moreover, $\lambda(0)=\Lambda^{-2}$. By assumption we have $A(0)\ge 0$. By the weak maximum principle for tensor \cite[Theorem A.21]{Chow-part1}, we conclude that the Theorem is true.
\end{proof}

\begin{rem}
 \begin{enumerate}
   \item[(i)]
In Theorem \ref{prop:metric-lowerBound}, if $g(t)$ is the Ricci flow, then $k=-\Ric(g)$ and $k+\Ric(g)=0$. $g(t)$ is a fixed metric, which means that  $q=0$, and $g$ has non-negative Ricci curvature, then we also have $k+\Ric(g)\ge0$. Hence the theorem can be applied to these cases. We may have corresponding results under the assumption that $k+\Ric\ge-K$, for $K\ge0$.
\item[(ii)] If the sectional curvature of $h$ is non positive, then we have $(F(t))^*(h)\le \Lambda^2 g(t)$ for all $g$.
\item[(iii)]  If $f$ is a $\Lambda$-Lipschitz map, then $\tr_{g}(f^*(h))=e(f)\le n\Lambda^2$. We may instead assume $\tr_g(f^*(h))\le n\Lambda^2$ and study the behaviour under the flow. The bound of $\tr_{g(t)}(F^*(h))$ is well-known. However, one may get the following sharp bound too. Namely, if the sectional curvature of $h$ is bounded above by $\kappa$, then $ (n\Lambda^2)^{-1}-2(n-1)\kappa t)\tr_{g(t)}(F^*(h))\le n$. One may wonder if there are similar relations of other symmetric functions of eigenvalues of $F^*(h)$.
 \end{enumerate}
\end{rem}

In the content of Ricci flow (i.e. $k=-\Ric$),  Theorem~\ref{prop:metric-lowerBound} is sharp when comparing with the following well-known scalar curvature estimate.
\begin{thm}\label{Thm: scalar curvatureLowerbdd}
Suppose $(M,g_0)$ is a compact manifold such that $\mathcal{R}(g_0)\geq \sigma_0$ on $M$ for some $\sigma_0>0$. If $g(t)$ is a smooth solution to the Ricci flow  on $M\times [0,T]$, then $T< \frac{n}{2\sigma_0}$. Moreover for all $t\in [0,T]$,
$$\mathcal{R}(g(t))\geq \frac{n\sigma_0}{n-2\sigma_0 t}.$$
\end{thm}
\begin{proof}
It is well-known that the scalar curvature evolves by
\begin{equation}
\heat \mathcal{R}=2|\Ric|^2\geq \frac2n \mathcal{R}^2.
\end{equation}
The desired estimate follows from maximum principle. The estimate on the existence time follows from the smoothness of the flow.
\end{proof}

\section{harmonic map heat flow from Lipschitz initial data}\label{Sec: HMHF-roughtData}

In Theorem \ref{intro-mainTHM}, we want to study continuous metric $g_0$ in the domain and Lipschitz continuous map $f$ from the domain manifold into the a compact manifold with smooth metric. We want to use the results in the previous section. Hence we want to construct Ricci flow with initial data $g_0$ and harmonic map heat flow with initial map $f$. We will do this by approximation. In this section, we will first construct harmonic heat flow coupled with  Ricci flow assuming that $g_0$ is smooth and $f$ is Lipschitz.
 To be precise, let $g(t)$ be a smooth solution of Ricci flow on $M\times [0,T]$:
\begin{equation}\label{RF-smooth-initial}
\left\{
\begin{array}{ll}
\partial_t g(t)=-2\Ric(g(t));\\
g(0)=g_0.
\end{array}
\right.
\end{equation}
Let $(N, h)$ be another compact manifold with smooth metric, and let  $f: (M,g_0)\to (N,h)$ be Lipschitz continuous.

\begin{thm}\label{thm: existence-RHF}
Let $(M,g_0)$ and $(N,h)$ be two smooth compact Riemannian manifolds. Suppose the sectional curvature of $h$ is bounded from above by $1$ and $g(t)$ is a solution to \eqref{RF-smooth-initial} on $M\times [0,T]$ such that $$|\Rm(g(t))|\leq at^{-1}$$
for some $a>0$ on $(0,T]$. If $f:(M,g_0)\to (N,h)$ is a $\Lambda$-Lipschitz map. Then there exists $F(t)\in C^\infty(M,N),\;t\in (0, \min\{T,T_0\}  ]$ satisfies the harmonic map heat flow
$$
\p_t F=\tau(F)
$$
where $\tau(F)$ is the tensor field of the map $F(\cdot,t)$ with respect to the metrics $g(t), h$ such that

\begin{equation}
\left\{
\begin{array}{ll}
\lf(\Lambda^{-2}-2(n-1)t\ri)(F(t))^*h \leq g(t);\\
\sup_{x\in M}d_h\left( F(x,t),f(x)\right)\leq C_0\sqrt{t}
\end{array}
\right.
\end{equation}
for some $T_0(n,\Lambda),C_0(n,a,\Lambda,h)>0$. Moreover, for any integer $\ell\ge0$, there exists $C(n,\ell,a,h,\Lambda)>0$ such that for all $t\in (0,\min\{T,T_0\}]$,
$$|D^\ell  dF|^2\leq \frac{C(n,\ell,a,h,\Lambda)}{t^{\ell}}.$$
\end{thm}
We should remark that the existence time and the estimates of $F$ do not depend on the curvature of the initial metric $g_0$, as long as $|\Rm(g(t))|\leq at^{-1}$ is true. This is important in application.

The theorem will follow from the results for smooth $f$ by approximation. First we have the following:
\begin{lma}\label{lma:approximation}
Suppose $(M ,g)$ and $(N,h)$ are smooth compact Riemannian manifolds and $f:(M,g)\to (N,h)$ is a continuous map such that $f$ is Lipschitz continuous. Then for any $\e>0$, there exists a smooth map $f_\e:(M,g)\to (N,h)$ such that
$$\mathrm{Lip}_{g,h}(f_\e)\leq (1+\e)\mathrm{Lip}_{g,h}(f)\quad\text{and}\quad \sup_{x\in M} d_h\left(f_\e(x),f(x) \right)<\e.$$
\end{lma}
\begin{proof} This is basically a well-known result by Greene-Wu \cite{GreeneWu1973}, see also \cite[Theorem 1.3]{KondoTanaka2017}. For completeness, we sketch the proof. First we isometrically embed $N$ to $\R^K$ for some large $K>0$. In this setting, the map $f$ can be expressed as a vector valued function $\mathbf{u}: M\to N\subset \R^K$. It is easy to see that $\mathbf{u}$ is   Lipschitz  with Lipschitz constant $L=: \mathrm{Lip}_{g,h}(f)$. For any $\eta>0$ suppose we can find a smooth function $\mathbf{v}$ such that $\sup_{x\in M}|\mathbf{u}(x)-\mathbf{v}(x)|\le \eta$ and $|D\mathbf{u}(w)|\le (L+\eta)|w|_g$ for all $w\in T(M)$, then the result follows. In fact, it this is true, then $\mathbf{v}(x)$ will be in the $\eta$ neighbourhood of $N$ in $\R^K$. If   $\pi\circ \mathbf{v}$ will be the required map if $\eta$ is small enough, where $\pi$ is the nearest point projection from $\R^K$ to $N$.

To find $\mathbf{v}$, we follows    the idea in \cite{GreeneWu1973}.
 Let $\rho: \R\to \R$ be a
non-negative smooth  function having support contained in $[-1, 1]$, so that
$$
\int_{\R^n} \rho(|x|) dx=1.
$$
Define
\be\label{e-convolution}
\mathbf{u}_\e (p)=\e^{-n}\int_{v\in T_p(M)}\mathbf{u}(\exp_p(v))\cdot \rho\lf(\frac{|v|}\e\ri)dV_p.
\ee
$dV_p$ is the volume form with respect  to $g|_{T_p(M)}$.
For $\e$ small enough, $\mathbf{u}_\e$ is smooth and $\bf{u}_\e\to \mathbf{u}$ uniformly as $\e\to0$. Let $p\in M$ and let $w\in T_p(M)$ with $||w||=1$. Let $\gamma(t)$ be the unique minimal geodesic with $\gamma(0)=p$, $\gamma'(0)=w$, $0\le t\le a$. We may fix $a$ so that $B_x(2a)$ is convex for all $x\in M$. Then
$$
 D\mathbf{u}_\e(w)=\lim_{t\to0}\frac1t\lf(\mathbf{u}_\e(\gamma(t))-\mathbf{u}_\e(\gamma(0)\ri).
$$

Now,
\bee
\begin{split}
&\mathbf{u}_\e(\gamma(t))-\mathbf{u}_\e(\gamma(0)\\
=&\e^{-n}\int_{v\in T_{\gamma(t)}(M)}\mathbf{u}(\exp_{\gamma(t)}(v))\cdot \rho\lf(\frac{|v|}\e\ri)dV_{\gamma(t)} -\e^{-n}\int_{v\in T_p(M)}\mathbf{u}(\exp_p(v))\cdot \rho\lf(\frac{|v|}\e\ri)dV_p\\
=&\e^{-n}\int_{v\in T_{p}(M), |v|\le \e}\lf(\mathbf{u}(\exp_{\gamma(t)}(P_t(v)))- \mathbf{u}(\exp_p(v))\ri) \rho\lf(\frac{|v|}\e\ri)dV_p
\end{split}
\eee
where $P_t$ is the parallel translation along $\gamma$ from $\gamma(0)$ to $\gamma(t)$. Hence
by Minkowski integral inequality \cite[Theorem 202]{Polya}, we have

\bee
\begin{split}
&|\mathbf{u}_\e(\gamma(t))-\mathbf{u}_\e(\gamma(0)|\\
\le &\e^{-n}\int_{v\in T_{p}(M), |v|\le \e}|\lf(\mathbf{u}(\exp_{\gamma(t)}(P_t(v))- \mathbf{u}(\exp_p(v))\ri)| \rho\lf(\frac{|v|}\e\ri)dV_p\\
\le &\sup_{v\in T_p(M), |v|\le \e}|\lf(\mathbf{u}(\exp_{\gamma(t)}(P_t(v))- \mathbf{u}(\exp_p(v))\ri)|\\
\le &L\sup_{v\in T_p(M), |v|\le \e}|d_g \lf(\mathbf{u}(\exp_{\gamma(t)}(P_t(v)), \mathbf{u}(\exp_p(v))\ri)|.
\end{split}
\eee

For each $t$, let $\gamma(t,s)$ be the geodesic from $\gamma(t)$ with tangent vector $P_t(v)$, $0\le s\le 1$. Then $|\p_t \gamma|_{s=0}=1$. Hence for any $\eta>0$, there is $\e>0$ such that if $|v|\le \e$, then $|\p_t\gamma|\le 1+\e$ for all $0\le s\le 1$. By compactness and the fact that $\exp$ is smooth, and   solutions of ODE depends smoothly on initial data, one can see that $\e>0$ can be chosen for all $p, w\in T_p(M)$ with $|w|=1$ and $v\in T(M)$ with $|v|\le \e$ we have $|\p_t\gamma|\le 1+\e$.
This implies that the length of the curve $\gamma(1,t)$, $0\le t\le t_0$ is less than or equal to $(1+C\e^2)t_0$. So
$$
d_g \lf(\mathbf{u}(\exp_{\gamma(t)}(P_t(v)), \mathbf{u}(\exp_p(v))\ri)\le (1+\e)t.
$$
From this we conclude that $ |D\mathbf{u}_\e(w)|\le L(1+\e)$ and the   result follows.
\end{proof}

\begin{proof}[Proof of Theorem \ref{thm: existence-RHF}] Let $L=\mathrm{Lip}_{g_0,h}(f)$ so that $L\leq \Lambda$. By Lemma \ref{lma:approximation}, there exist smooth maps $f_i$ from $M$ to $N$ so that $$\mathrm{Lip}_{g_0,h}(f_i)\leq (1+\e_i)L\quad\text{and}\quad \sup_{x\in M} d_h\left(f_i(x),f(x) \right)<\e_i$$ with $\e_i\to 0$ as $i\to\infty$. In particular, the energy density of the maps $f_i$ are uniformly bounded by $(n+1)L^2$, say if $i$ is sufficiently large. By \cite[Theorem 1.1]{HuangTam2021}, for each $i$ there is a solution $F_i$ to the harmonic heat flow coupled with $g(t)$ on $M\times[0,\min\{T,T_0\}]$
for some $T_0(n,\Lambda)$ so that for any integer $\ell\ge0$, there exists $C(n,\ell,a,h,\Lambda)>0$ such that for all $t\in (0,\min\{T,T_0\}]$,

\be\label{e-derivativesF}
|D^\ell  dF_i|^2\leq \frac{C(n,\ell,a,h,\Lambda)}{t^{\ell}}.
\ee
In \cite{HuangTam2021}, only the energy density and the tension field have been estimated. But the higher order estimates of $|D^\ell d F_i|^2$ can be done similarly using Bernstein-Shi trick as in \cite{ChenZhu2006}.  The only difference is that in \cite{ChenZhu2006}, $|\Rm(g(t)|$ is uniformly bounded and in our case, $|\Rm(g(t))|\le at^{-1}$. So we give a sketch of proof here. By the estimates of Shi \cite{Shi1989}, $|\nabla^k\Rm(g(t))|\le Ct^{-(1+\frac k2)}$ for any integer $k\ge0$, where $C$ depends only on $a, n, k$.
Denote $|D^\ell dF_i|^2$ by $P_\ell$, where $P_0$ is just the energy density. So \eqref{e-derivativesF} is true for $\ell=0$. Suppose the estimates are true up to $\ell-1$. We will use $C_i$ to denote any constants depending only on $n,T,a,\Lambda,h,\ell$.  By a direct computation \cite[Lemma 2.10 and p.141]{ChenZhu2006}, we have
\bee
\heat P_\ell\le -2P_{\ell+1}+C_1\lf(t^{-1}P_\ell+t^{-1-\frac\ell2}P_\ell^\frac12\ri).
\eee
Hence whenever $P_{\ell}>0$, we have
\bee
\heat (t^{ \frac{1+\ell}2}P^\frac12_\ell) \le C_3(t^{\frac{\ell-1}2}P^\frac12_\ell+t^{-\frac12})
\eee
On the other hand,
\bee
\begin{split}
\heat P_{\ell-1}\le &-2P_{\ell}+C_4t^{-\ell}
\end{split}
\eee
and so
\bee
\heat (t^{\ell-\frac12}P_{\ell-1})\le-2t^{\ell-\frac12}P_\ell+C_5t^{-\frac12}.
\eee
Let $G=t^{\frac{1+\ell}2}P^\frac12+t^{\ell-\frac12}P_{\ell-1}-\a t^\frac12$, for $\a>0$, we have:
\bee
\begin{split}
\heat G\le & C_3(t^{\frac{\ell-1}2}P^\frac12_\ell+t^{-\frac12})-2t^{\ell-\frac12}P_\ell+C_5t^{-\frac12}-\frac12\a t^{-\frac12}\\
\le &C_6t^{-\frac12}-\frac12\a t^{-\frac12}.
\end{split}
\eee
Let $\a$ be such that $\frac12\a=C_1+1$. We have
$$
\heat G<0
$$
at the points where $P_\ell>0$. Since $G=0$ at $t=0$, we conclude that $G\le 0$.
Hence
$$
t^{\frac\ell 2}P_{\ell}^\frac12\le C_7.
$$
This completes the proof of the estimates \eqref{e-derivativesF}.

Since $|\tau(F_i)|\le |DdF_i|\le Ct^{-\frac12}$, we conclude that
$$
d_h(F_i(x,t),f(x))\le C_8t^\frac12.
$$
 Since $N$ is compact, we can find a subsequence of $F_i$ which converges uniformly in $C^\infty$ to a map $F$ in compact subsets of $M\times(0,\min\{T,T_0\}]$. Hence $F$ satisfies the harmonic heat flow with estimates of $|D^\ell dF|$ in the theorem. Moreover,
$$
d_h(F(x,t),f(x))\le C_8t^\frac12.
$$

Finally, apply Theorem \ref{prop:metric-lowerBound} and
pass it to limit, we have
$$
(\Lambda^{-2}-2(n-1)t)(F(t))^*h\le g(t).
$$
This completes the proof of the theorem.

\end{proof}

\section{Proof of Theorem~\ref{intro-mainTHM}}\label{Sec:Proof-Main}

In this section, we will construct a solution to the harmonic map heat flow coupled with the Ricci flow where both $f$ and $g_0$ are allowed to be non-smooth. Due to the weak parabolicity of the Ricci flow,  there will be some technical issue when discussing the time zero regularity of the Ricci flow if the initial data is only $C^0$. To avoid this, we will work on the Ricci-Deturck $G$-flow instead, where $G$ is a smooth background metric.

A smooth family of metrics $ g(t)$ on $M\times (0,T]$ is said to be a solution to the $G$-flow if it satisfies
\begin{equation}\label{equ:h-flow}
\left\{
\begin{array}{ll}
\partial_t g_{ij}=-2R_{ij}+\nabla_i W_j +\nabla_j W_i;\\
W^k= g^{pq}\left( \Gamma^k_{pq}-\wt\Gamma^k_{pq} \right)
\end{array}
\right.
\end{equation}
where $\Gamma$ and $\wt\Gamma$ denote the connections of $g(t)$ and $G$ respectively.
If the initial metric $g_0$ is smooth, it is well-known that the Ricci flow is equivalent to the Ricci-Deturck $G$-flow in the following sense. Let $\Phi_t$ be the diffeomorphism given by
\begin{equation}\label{e-Phi}
\left\{
\begin{split}
\partial_t\Phi_t(x)&=-W\left(\Phi_t(x),t \right);\\
\Phi_0(x)&=x.
\end{split}
\right.
\end{equation}
Then the pull-back of the Ricci-Deturck flow $ \wh g(t)=\Phi_t^*  g(t)$ is a Ricci flow solution with $\wh g(0)= g(0)=g_0$.

\begin{proof}[Proof of Theorem~\ref{intro-mainTHM}] Let $g_0$ be a $C^0$ metric on $M$ with scalar curvature at least $n(n-1)$ in the sense of Definition \ref{defn-C0-scalar}. Then there exist a sequence of smooth metrics $g_{i,0}$ with $\cR(g_{i,0})\ge n(n-1)$ so that $g_{i,0}\to g_0$  in $C^0$-norm. We may assume that
\be\label{e-limitmetrics}
\left(1+\frac1i \right)^{-1}g_{i,0}\le g_0\le \left(1+\frac1i \right)g_{i,0}.
\ee

Fix $i_0$ large enough  only on $n$ and let $G=:g_{i_0}$ so that we can apply the results by  M. Simon \cite{Simon2002} to obtain a constant  $T>0$ depending only on $n$ and $G=:g_{i_0}$ such that the $G$-flow \eqref{equ:h-flow} with initial data $g_{i,0}$ has a solution defined on $M\times[0,T]$ which is smooth up to $t=0$. Denote this solution by $g_i(t)$. Moreover,
by \cite{Simon2002}, for any $\ell\in \mathbb{N}$, there is $C_\ell>0$ depending only on $n,k_0,...,k_{\ell+1}$ so that
\be\label{e-derivativesbdd}
\left\{
  \begin{array}{ll}
    \frac12(1+\frac1{i_0})^{-1}g_i\le G\le 2(1+\frac1{i_0}) g_i;\\
\sup_M |\wt \nabla^\ell g_i(t)|\leq \displaystyle\frac{C_\ell}{t^{\ell/2}},
  \end{array}
\right.
\ee
where $\wt \nabla$ is the covariant derivative with respect to $G$ and $k_m$ is the bound of the norm of the $m$-th derivative  of the curvature tensor of $G$. Passing to some subsequence, we may assume that $g_i(t)$ will converge to a solution $g(t)$ to the $G$-flow on $M\times(0, T]$. Moreover, by the work of Simon \cite{Simon2002}, $g(t)$ converges to $g_0$ in $C^0$ norm as $t\to0$. Since $\cR(g_{i,0})\ge n(n-1)$,  Theorem \ref{Thm: scalar curvatureLowerbdd} implies
\be\label{e-scalar}
\cR\lf(\lambda^{-1}(t)g(t)\ri)\ge  n(n-1).
\ee
where $\lambda(t)=1-2(n-1)t$ as long as $\lambda(t)>0$. We may assume that this is true for $0\le t\le T$ by choosing a smaller $T>0$. We denote $\check g(t)=\lambda(t)^{-1}g(t)$.

Next we want to construct "harmonic" map heat flow coupled with $g(t)$. For each $i$, let $W^{(i)}$ be the vector field  given by \eqref{equ:h-flow} with $\Gamma=\Gamma^{(i)}$ to be the connection of $g_i(t)$ and let $\Phi_{t}^{(i)} $ be as in \eqref{e-Phi} with respect to $W^{(i)}$. Let $\wh g_i(t)=(\Phi_{t}^{(i)} )^*(g_i(t))$. Then $\wt g_i(t)$ is a smooth solution of the Ricci flow in $[0,T]$ such that $\wh g_i(0)=g_{i,0}$. By \eqref{e-limitmetrics} and the assumption that $f$ is a 1-Lipschitz map from $(M,g_0)$ to $(N,h)$, we conclude that $f$ is a $(1+\frac1i)$-Lipschitz maps from $(M,g_{i,0})$ to $(N,h)$. By \eqref{e-derivativesbdd}
and the fact that $\wh g_i(t)$ is isometric to $g_i(t)$, we conclude that $|\Rm(\wh g_i(t)|\le at^{-1}$ on $M\times(0,T]$ for some $a>0$ independent of $i$. By Theorem \ref{thm: existence-RHF}, we can solve the harmonic heat flow $\wh F^{(i)}(t)$ coupled with the Ricci flow $\wh g_i(t)$ in $M\times [0,T]$ (with a possibly smaller $T$ but  independent of $i$) with the following properties:

\begin{enumerate}
  \item [(i)] $\wh  F^{(i)}(t)$ is a $L^i(t)$-Lipschitz map from $(M, \wh g_i(t))$ to $(N,h)$, with
      $$
      L^i(t)=\frac{(1+\frac1i)^2}{1-2(n-1)(1+\frac1i)^2t}.
$$
  \item [(ii)] For each $\ell\ge0$,   $|D^{\ell}d \wh F^{(i)}(t)|^2\le C_{\ell}t^{-\ell}$.
      \item[(iii)] $\wh  F^{(i)}(0)=f$.
\end{enumerate}

Let $F^{(i)}(x,t)=\wh  F^{(i)}\left(\left(\Phi_t^{(i)}\right)^{-1}(x),t\right)$. Then  using \eqref{e-Phi}, $F^{(i)}$ satisfies:
\begin{equation}\label{e-Fi}
\begin{split}
\partial_t F^{(i)}(x,t)&= \partial_t \left(\hat F^{(i)} ((\Phi^{(i)}_t)^{-1}(x)  \right)\\
&=\tau( \wh F^{(i)})+d\wh F^{(i)}\left(\partial_t (\Phi^{(i)}_t)^{-1}(x)  \right)\\
&= \tau(\wh F^{(i)}) +d\hat F_i(W ^{(i)})| \left(\Phi_{i,t}^{-1}(x)\right)\\
&=\tau( F^{(i)}) +dF^{(i)}(W^{(i)}).
\end{split}
\end{equation}
Moreover, since $\Phi^{(i)}_t$ is an isometry between $\wh g_{i}(t)$ and $g_i(t)$, we know that
\begin{enumerate}
  \item [(i)] $F^{(i)}(t)$ is a $L^i(t)$-Lipschitz map from $(M,  g_i(t))$ to $(N,h)$.
  \item [(ii)] For each $\ell\ge0$,   $|D^{\ell}d   F^{(i)} (t)|^2\le C_{\ell}t^{-\ell}$ where  $D$ is the covariant derivative with respect to $g_i(t)$ and $h$ on the pull-back bundle $(F^{(i)})^{-1}(T(N))$.
      \item[(iii)] $F^{(i)}(0)=f$ because $\Phi^{(i)}_0(x)=x$.
\end{enumerate}
By \eqref{e-Fi}, $|\tau(F^{(i)})|_h $ is bounded by $Ct^{-\frac12}$, $|dF^{(i)} |_{g_i(t),h}$    is bounded uniformly and $|W_i|_{g_i(t)}\leq Ct^{-1/2}$,  for some $C$ independent of $i, x, t$ using \eqref{e-derivativesbdd}.  By computing the length of the curve $F(x,s): 0\le s\le t$ in $N$,  we conclude  that for all $x\in M$,
\begin{equation}\label{e-distance}
d_h\left(F_i(x,t),f(x) \right)\leq  C(n,h,G) \sqrt{t}.
\end{equation}
  On the other hand, by (ii) above,  \eqref{e-derivativesbdd} and passing to a subsequence, $F_i$ will converges to a smooth map $F(x,t)$ from $M\times(0,T_1)$ to $(N,h)$ for some $T_10$.  Hence
  $F(t)$ is a 1-Lipschitz map from $(M,\check g(t))$ to $(N,h)$. The scalar curvature of $\check g(t)$ is at least $n(n-1)$ by Theorem \ref{Thm: scalar curvatureLowerbdd}. On the other hand, by \eqref{e-distance}, we conclude that $\sup_M d_h(F(x,t),f(x))\to 0$ as $t\to 0$. Hence $F(x,t)$ will be homotopy to $f(x)$ if $t$ is small enough. This implies $F(x,t)$ has non-zero degree.
  Now the result of Llarull \cite{Llarull1998} applies to show that $F(t)$ is a metric isometry and $\check g(t)$ is the standard sphere via $F(t)$ so that for all $x,y\in M$ and $t\in (0,T]$,
$$d_h\left( F_t(x),F_t(y)\right)=d_{\check g(t)}(x,y).$$
By letting $t\to 0$ using the fact that $g(t)$ converges uniformly to $g_0$ as $t\to 0$,   we conclude that $f$ is a distance isometry.  This completes the proof.
\end{proof}
\bigskip

With the Ricci flow smoothing, we can slightly improve the result in \cite{CecchiniHankeSchick2022} when $g_0$ is of slightly better regularity and of scalar curvature lower bound in the sense of Lee-LeFloch \cite{LeeLeFloch2015}. We recall the definition of distributional scalar curvature lower bound first.

\begin{defn}[\cite{LeeLeFloch2015}]\label{R-lower-dist}
Let $M$ be a smooth manifold with background metric $h$. A $W^{1,2}_{loc}\cap L^\infty$ Riemannian metric $g$ is said to have $\mathcal{R}(g)\geq \sigma$ for $\sigma\in \mathbb{R}$ in the sense of distribution if $\langle \mathcal{R}(g),\varphi\rangle \geq \sigma \int_M \varphi \; d\mathrm{vol}_h$ for all non-negative test function $\varphi\in C^\infty_{loc}(M)$ where
\begin{equation}
\begin{split}
\langle \mathcal{R}(g_0),\varphi\rangle=\int_M -\left\langle V,\tilde \nabla (\varphi \cdot \sqrt{\frac{\det g_0}{\det h}})\right\rangle_h+ F\varphi  \cdot \sqrt{\frac{\det g_0}{\det h}}\; d\mathrm{vol}_h
\end{split}
\end{equation}
with
\begin{equation}
\left\{
\begin{array}{ll}
\Psi^k_{ij}=\frac12 g^{kl} \left(\tilde \nabla_i g_{kl}+\tilde\nabla_j g_{il}-\tilde\nabla_l g_{ij} \right);\\
V^k=g^{ij}\Psi_{ij}^k -g^{ik}\Psi_{ji}^j;\\
F=\tr_g \tilde\Ric - \Psi_{ij}^k\tilde\nabla_k g^{ij} + \Psi^i_{jl}\tilde\nabla_k g^{ik}+g^{ij} \left(\Psi^k_{kl}\Psi^l_{ij} -\Psi^k_{jl}\Psi^l_{ik} \right)
\end{array}
\right.
\end{equation}
Here $\tilde\nabla$ denotes the connection with respect to the background metric $h$.
\end{defn}

With the slightly stronger assumption, we might conclude a slightly better regularity of $g_0$.
\begin{thm}\label{Thm:better-regular}
Under the assumption in Theorem~\ref{intro-mainTHM}, if in addition $g_0\in W^{1,p}(M),p>n$ satisfies $\mathcal{R}(g_0)\geq n(n-1)$ in the sense of Lee-LeFloch, then there exists a $C^{1,\a}$ diffeomorphism $\Psi$ of $M$ for some $\a>0$ so that $\Psi^*g_0$ is the standard spherical metric on $\mathbb{S}^n$.
\end{thm}
\begin{proof}
The main result follows from refining the proof of Theorem~\ref{intro-mainTHM}.  The idea is similar to that in \cite{LeeTam2021}. Since $g_0\in W^{1,p}$ for $p>n$,  the works in \cite{JiangShengZhang2021,ShiTam2016} infers that the  constructed $G$-flow $g(t)$ satisfies a better estimate on $M\times (0,T]$:
\be\label{Improve-e-derivativesbdd}
\int_M |\tilde \nabla g(t)|^p \leq C.
\ee

Moreover, it is known that each  $\check g(t)=\lambda^{-1}(t)g(t),t\in (0,T]$ is isometric to the standard sphere.  We now construct the diffeomorphism to compensate the singularity of \eqref{Improve-e-derivativesbdd} as $t\to 0$.   Using \eqref{equ:h-flow} and $\Ric(\check g(t))=n-1$, we deduce that
\begin{equation}
\left\{
\begin{array}{ll}
\partial_t \check g_{ij}=\check  \nabla_i \check W_j+\check  \nabla_j \check W_i;\\
\check W_j=\lambda(t)^{-1} \check g_{jk} \check g^{pq}\left( \check \Gamma_{pq}^k-\tilde\Gamma_{pq}^k\right)
\end{array}
\right.
\end{equation}

Therefore, if we consider the ODE:
\begin{equation}
\left\{
\begin{array}{ll}
\partial_t \Psi_t(x)= \check W\left( \Psi_t(x),t\right);\\
\Psi_T(x)=x
\end{array}
\right.
\end{equation}
for $(x,t)\in M\times (0,T]$, then $$\partial_t\left[ (\Psi_t^{-1})^* \check g(t)\right]=0.$$
Hence,  $\check g(t)=\Psi_t^* \check g(T)$ where $\check g(T)$ is a standard spherical metric.  By \cite[(5.2)]{CalabiHartman1970},
\begin{equation}\label{local-DIFF-CalabiHartman}
\frac{\partial^2 \Psi_t^m}{\partial x^i \partial x^j} =\check \Gamma^k_{ij} \frac{\partial \Psi_t^m}{\partial x^k} -\Gamma(\check g(T))_{kl}^m \frac{\partial \Psi_t^l}{\partial x^i}\frac{\partial \Psi_t^k}{\partial x^j}
\end{equation}
in local coordinate of $M$.  Thanks to the improved estimate \eqref{Improve-e-derivativesbdd} and metrics equivalence,  the right hand side of \eqref{local-DIFF-CalabiHartman} is bounded in $L^p,p>n$ as $t\to 0$ and hence $\Psi_t\in W^{2,p}$ uniformly in $t\to 0$.  By standard Sobolev embedding,  we may pass $\Psi_t\to \Psi_0$ for $t\to 0$ in $C^{1,\a}$ for some $\a>0$ and $C^{1,\a}$ diffeomorphism $\Psi_0$ of $M$.  This completes the proof as $\check g(t)\to g_0$ as $t\to 0$.
\end{proof}
\bigskip


As an application of Theorem~\ref{intro-mainTHM}, we prove a rigidity concerning the domain inside sphere.  To make our statement precise,  for a domain $\Omega$ in a Riemannian manifold, we shall adopt the convention $H=\tr(\nabla \nu)$ where $\nu$ is the outward unit normal along $\partial\Omega$ so that the mean curvature of boundary of unit ball is $n-1$. Now we are ready to prove Corollary~\ref{c-domain}.

\begin{proof}[Proof of Corollary~\ref{c-domain}]
The strategy is similar to that in \cite[Theorem B]{CecchiniHankeSchick2022}.  We first consider the general case. Consider the metric $ g$ on $\mathbb{S}^n$ defined by
\begin{equation}
g=\left\{
\begin{array}{ll}
g_0\quad\text{on}\quad \overline\Omega;\\
h\quad\text{on}\quad \mathbb{S}^n\setminus \Omega.
\end{array}
\right.
\end{equation}

With this identification, the metric $g$ is Lipschitz on $M$ and \cite[Proposition 5.1]{LeeLeFloch2015} applies to show that $\mathcal{R}(g)\geq n(n-1)$ in the distribution sense. And hence, $g$ satisfies the assumptions of Theorem~\ref{intro-mainTHM} by \cite[Corollary 1.2]{JiangShengZhang2021} with $f$ chosen to be identity.  Hence, $g=h$ and so does $g_0$ on $\Omega$.

If $\Omega$ is the hemisphere,  then we have $H(g)\geq 0$. Now we instead define $g$ by taking the reflected metric on $\mathbb{S}^n\setminus \Omega$. In this case, $g$ is also Lipschitz and satisfies $\mathcal{R}(g)\geq n(n-1)$ in the distribution sense by \cite[Proposition 5.1]{LeeLeFloch2015}. The conclusion follows using Theorem~\ref{intro-mainTHM}.
\end{proof}


\begin{thebibliography}{100}


\bibitem{BaerHanke2020} Baer,  C.; Hanke, B.,  {\sl Boundary conditions for scalar curvature}, arXiv:2012.09127

\bibitem{Bamler2016}Bamler, R., {\sl A Ricci flow proof of a result by gromov on lower bounds for scalar curvature}, Math. Res. Lett. 23 (2016), no. 2, 325--337.



\bibitem{Burkhardt2019}Burkhardt-Guim, P., {\sl Pointwise lower scalar curvature bounds for $C^0$ metrics via regularizing Ricci flow}, Geom. Funct. Anal. 29 (2019), no. 6, 1703--1772.
    
    \bibitem{Burtscher2015} Burtscher, A., {\sl Length structures on manifolds with continuous Riemannian metrics}. New York J. Math., 21:273--296, 2015.


\bibitem{CalabiHartman1970}Calabi, E.; Hartman, P.,  {\sl On the smoothness of isometries}, Duke Math. J. 37 (1970), 741--750.

\bibitem{CecchiniHankeSchick2022} Cecchini, S.; Hanke, B.; Schick, T., {\sl Lipschitz rigidity for scalar curvature}, arXiv:2206.11796

\bibitem{CecchiniZeidler2021} Cecchini, S. ; Zeidler, R. , {\sl Scalar and mean curvature comparison via the Dirac operator} , arXiv:2103.06833, to appear in Geom. Topol.




\bibitem{ChenZhu2006}Chen, B.-L.;, Zhu, X.-P., {\sl Uniqueness of the Ricci flow on complete noncompact manifolds}, J. Differential Geometry 74 (2006), 119--154.

\bibitem{Chow-part1} Chow, B. et al, {\sl The Ricci flow: techniques and applications, Part I. Geometric aspects}. Math. Survey and Monographs, 135, Amer. Math. Soc., Prodidence, RI, (2007).


\bibitem{Polya}Hardy, G. H. ,  Littlewood, J.E. ; Polya G., {\sl Inequalities}, Cambridge Univ. Press, Cambridge (1934).


\bibitem{HuLiuShi2022} Hu, Y.; Liu,  P.; Shi, Y., {\sl Rigidity of 3D spherical caps via $\mu$-bubbles}, arXiv:2205.08428
    
    \bibitem{HuangLee2021} Huang, Y.; Lee, M.-C.,  {\sl Scalar curvature lower bound under integral convergence}, arXiv:2111.05079
    
\bibitem{HuangTam2021}Huang, S.; Tam, L.-F., {\sl Short time existence for harmonic map heat flow with time-dependent metrics}, arXiv:2110.07142, to appear in J. Geom. Anal.


\bibitem{GoetteSemmelmann2002}Goette, S.; Semmelmann,  U.,{\sl Scalar curvature estimates for compact symmetric spaces}, Diff. Geom. Appl. 16, p. 65-78 (2002)

 \bibitem{GreeneWu1973} Greene, R. E.; Wu, H., {\sl On the subharmonicity and plurisubharmonicity of geodesically convex functions}, Indiana Univ. Math. J. {\bf 22} (1972/73), 641--653.


\bibitem{Gromov2014}Gromov, M., {\sl Dirac and Plateau billiards in domains with corners}, Cent. Eur. J. Math. 12 (2014), no. 8, 1109--1156.


\bibitem{GromovLecture} Gromov, M. , {\sl Four Lectures on Scalar Curvature}, arXiv:1908.10612







 \bibitem{JiangShengZhang2021} Jiang, W.; Sheng, W.;  Zhang, H., {\sl Weak scalar curvature lower bounds along Ricci flow}, arXiv:2110.12157.

\bibitem{KochLamm2012}Koch, H.; Lamm, T., {\sl Geometric flows with rough initial data}. Asian J. Math. 16 (2012), no. 2, 209--235.

\bibitem{KondoTanaka2017}Kondo, K.; Tanaka, M., {\sl Approximations of Lipschitz maps via immersions and differentiable exotic sphere theorems}. Nonlinear Anal. 155 (2017), 219--249.


\bibitem{LeeLeFloch2015}Lee, D. A.; LeFloch, P. G., {\sl The positive mass theorem for manifolds with distributional curvature}, Comm. Math. Phys. 339 (2015), no. 1, 99--120.

\bibitem{LeeTam2021} Lee, M.-C. ; Tam, L.-F., {\sl Continuous metrics and a conjecture of Schoen},  arXiv:2111.05582


\bibitem{Llarull1998} Llarull, M., {\sl Sharp estimates and the Dirac operator}. Math. Ann., 310(1):55--71, 1998.

\bibitem{Lott2021} Lott, J. ,  {\sl Index theory for scalar curvature on manifolds with boundary.} Proc. Am. Math. Soc., 149(10):4451--4459, 2021.


\bibitem{Shi1989}Shi, W.-X., {\sl Deforming the metric on complete Riemannian manifolds}, J. Differential Geom. 30 (1989), no. 1, 223--301.


\bibitem{ShiTam2016}Shi, Y.; Tam, L.-F., {\sl Scalar curvature and singular metrics}, Pacific J. Math. 293 (2018), no. 2, 427--470.

%
\bibitem{Simon2002}Simon, M., {\sl Deformation of $C^0$ Riemannian metrics in the direction of their Ricci curvature}, Comm. Anal. Geom. 10(2002), no. 5, 10331074



%
%
%
%
%
%
%
%
%
%
%
%
%
%
%
%
%
%
%
%
%
%
%
%
%

%
%
%
%
%
%

%
%
%

%
%
%
%
%
%
%
%
%
%
%
%
%
%
%
%
%
%
%
%
%
%
%
%
%
%
%
%
%
%
%
%
%

%
%

%
%
%
%

\end{thebibliography}
\end{document}